\documentclass[amsthm,leqno]{monsky2009} % for arXiv.

\usepackage{graphicx}
\usepackage{amssymb}
\usepackage{amsmath}

\usepackage{bm}

\newtheorem*{corollary*}{Corollary}
\newtheorem*{remark*}{Remark}

\newtheorem{definition}{Definition}[section] %section number prepended.
\newtheorem{theorem}[definition]{Theorem}
\newtheorem{lemma}[definition]{Lemma}

\newtheorem{example}[definition]{Example}
\newtheorem{corollary}[definition]{Corollary}
\newtheorem{remark}[definition]{Remark}

\DeclareMathOperator{\hilb}{hilb}
\DeclareMathOperator{\poin}{poincar\acute{e}}

\newcommand{\dif}{\ensuremath{\mathrm{d}}}

\newcommand{\supcheck}{\ensuremath{\,\raisebox{-1ex}{$\textstyle\check{}\,$}}}
\newcommand{\lcheck}{\ensuremath{L^{\supcheck}}}

\newcommand{\newo}{\ensuremath{\mathcal{O}}}
\newcommand{\newp}{\ensuremath{\mathbb{P}}}

\newcommand{\newz}{\ensuremath{\mathbb{Z}}}
\newcommand{\newb}{\ensuremath{\mathcal{B}}}

\begin{document}

\begin{frontmatter}



%\title{}


%\author{}

\title{Hilbert-Kunz theory for nodal cubics, via sheaves}
\author{Paul Monsky}

\address{Brandeis University, Waltham MA  02454-9110, USA\\  monsky@brandeis.edu }

\begin{abstract}
Suppose $B=F[x,y,z]/h$ is the homogeneous coordinate ring of a characteristic $p$ degree $3$ irreducible plane curve $C$ with a node. Let $J$ be a homogeneous $(x,y,z)$-primary ideal and $n\rightarrow e_{n}$ be the Hilbert-Kunz function of $B$ with respect to $J$.

Let $q=p^{n}$. When $J=(x,y,z)$, Pardue (see \cite{3}) showed that $e_{n}=\frac{7}{3}q^{2}-\frac{1}{3}q-R$ where $R=\frac{5}{3}$ if $q\equiv 2 \pod{3}$, and is $1$ otherwise. We generalize this, showing that $e_{n}=\mu q^{2}+\alpha q-R$ where $R$ only depends on $q \mod 3$. We describe $\alpha$ and $R$ in terms of classification data for a vector bundle on $C$. Igor Burban \cite{4} provided a major tool in our proof by showing how pull-back by Frobenius affects the classification data of an indecomposable vector bundle over $C$. We are also indebted to him for pointing us towards \cite{5}, in which $h^{0}$ is described in terms of these classification data.

\end{abstract}
\maketitle
%\begin{keyword}

%\PACS 
%\end{keyword}
\end{frontmatter}

% main text

\section*{Introduction}

Let $h$ be a form of degree $>0$ in $A=F[x,y,z]$ where $F$ is algebraically closed of characteristic $p>0$. Suppose $J$ is a homogeneous ideal of $A$. If $q=p^{n}$, let $J^{[q]}$ be the ideal generated by all $u^{q}$, $u$ in $J$.  Let $e_{n}$ be the $F$-dimension of $A/(J^{[q]},h)$.

\subsection*{Problem: If $e_{0}<\infty$, how does $e_{n}$ depend on $n$?}

The problem was treated by elementary methods, when $J=(x,y,z)$ and degree $h$ is small, by several authors. In particular, Pardue in his thesis (see \cite{3} for an exposition) showed that when $h$ is an irreducible nodal cubic then $e_{n}$ is $\frac{7}{3}q^{2}-\frac{1}{3}q-\frac{5}{3}$ if $q\equiv 2\pod{3}$, and is $\frac{7}{3}q^{2}-\frac{1}{3}q-1$ otherwise.

For arbitrary $h$ and $J$, sheaf-theoretic methods were introduced by Brenner \cite{1} and Trivedi \cite{8}. They calculated $\mu=\lim_{n\rightarrow\infty}\frac{e_{n}}{q^{2}}$, showing that $\mu$ is rational. When $h$ has coefficients in a finite field and defines a smooth plane curve $C$, Brenner \cite{2} showed further that $\mu q^{2}-e_{n}$ is an eventually periodic function of $n$. In \cite{7}, the author returned to the case $J=(x,y,z)$, and adapted Brenner's method to treat all $h$ defining reduced irreducible $C$. (But now $\mu q^{2}$ must be replaced by something a bit more complicated.)

In the present paper we restrict our attention to nodal cubics but allow $J$ to be arbitrary. Using sheaf-theoretic methods as in \cite{7} we recover Pardue's result when $J=(x,y,z)$. For arbitrary $J$ we get a result nearly as precise. What allows us to get sharp results is the well-developed theory of vector bundles on nodal cubic curves. (See Igor Burban \cite{4} and the references therein.) We are indebted to Burban for pointing us towards this theory, and for the result essential to us that he derives in \cite{4}.

\section{A little sheaf theory}
\label{section1}

\begin{definition}
\label{def1.1}
If $M$ is a finitely generated $\newz$-graded $A=F[x,y,z]$ module, $\hilb (M) =\sum \dim (M_{d})T^{d}$ and $\poin(M)=(1-T)^{3}\hilb (M)$. (Note that $\poin(M)$ is in $\newz[T,T^{-1}]$.)
\end{definition}

Throughout the paper we adopt the notation of the introduction, with $h\in A$ a degree $3$ form defining a nodal $C\subset \newp ^{2}$, having desingularization $X=\newp^{1}$. Hartshorne \cite{5} is a good reference for what follows.

Even though $C$ is singular there is a good theory of torsion-free sheaves on $C$. One may define the degree of such a sheaf, all such sheaves are reflexive, and one has Riemann-Roch and Serre duality. In some ways $C$ is like an elliptic curve. For example, if $Y$ is rank $1$ torsion-free, $h^{0}(Y)=\deg Y$ if $\deg Y>0$, and is $0$ if $\deg Y<0$. When $\deg Y=0$, $h^{0}(Y)$ is $1$ if $Y$ is isomorphic to $O_{C}$ and is $0$ otherwise.

\begin{definition}
\label{def1.2}
$\poin(Y)=(1-T)^{3}\sum h^{0}(Y(n))T^{n}$, where $Y(n)$ is the twist of $Y$ by $O_{C}(n)$. (Riemann-Roch shows that $(1-T)^{-1}\poin(Y)$ is in $\newz[T,T^{-1}]$.)
\end{definition}

\begin{example}
\label{example1.3}
\begin{enumerate}\item[] %null item to cause newline
\item[(a)] $\poin(O_{C})=(1-T)^{3}(1+3T+6T^{2}+9T^{3}+\cdots ) = 1-T^{3}$
\item[(b)] $\poin(\oplus O_{C}(-d_{i}))=(1-T^{3})\cdot\sum T^{d_{i}}$
\item[(c)] If $L$ has rank $1$ and degree $-n$, then:

\begin{eqnarray*}
(1-T)^{-1}\poin(L) &=& T^{\frac{n+2}{3}}(2+T)\text{ if } n\equiv 1\pod{3}\\
 &=& T^{\frac{n+1}{3}}(1+2T)\text{ if } n\equiv 2\pod{3}\\
 &=& T^{\frac{n}{3}}(1+T+T^{2})\text{ if } L\approx O_{C}\left(-\frac{n}{3}\right)\\
 &=& T^{\frac{n}{3}}(3T)\text{ otherwise }
\end{eqnarray*}
\end{enumerate}

\end{example}

\begin{lemma}
\label{lemma1.4}
Suppose $L$ and $M$ are rank $1$ torsion-free, that neither is isomorphic to any $O_{C}(k)$, and that $\deg M \le 1+\deg L$. Then if $0\rightarrow L\rightarrow U \rightarrow M\rightarrow 0$ is exact, $\poin(U)=\poin(L)+\poin(M)$.
\end{lemma}

\begin{proof}
Since $\deg M(n)\le 1+\deg L(n)$ for each $n$, it's enough to show that $h^{0}(U)=h^{0}(L)+h^{0}(M)$. If $\deg L\ge 0$, $\deg \lcheck\le 0$ and $\lcheck$ is not isomorphic to $O_{C}$. So $h^{1}(L) = h^{0}(\lcheck)=0$, and we use the exact sequence of cohomology. If $\deg L < 0$, $\deg M \le 0$, and $M$ is not isomorphic to $O_{C}$. So $h^{0}(M)=0$, and the result follows.
\end{proof}

Now fix a homogeneous ideal $J$ of $A$ with $\dim A/(J,h) <\infty$, and forms $g_{1},\ldots , g_{s}$ generating $(J,h)/h$, with $\deg g_{i}=d_{i}$. Then the sheaf map\linebreak $\oplus O_{C}(-d_{i})\rightarrow O_{C}$ defined by the $g_{i}$ is onto. So if $W$ is the kernel of this map, $W$ is locally free of rank $s-1$ and degree $-3\sum d_{i}$.

\begin{lemma}
\label{lemma1.5}
\begin{enumerate}\item[] %null item to cause newline
\item[(1)] $\poin\left(A/(J,h)\right)=(1-T)^{3}\left(1-\sum T^{d_{i}}\right)+\poin(W)$
\item[(2)] More generally, let $q=p^{n}$ and $W^{[q]}$ be the pull-back of $W$ by $\Phi^{n}$, where $\Phi : C\rightarrow C$ is the Frobenius map. Then:
$$\poin\left(A/(J^{[q]},h)\right) = (1-T^{3})\left(1-\sum T^{qd_{i}}\right)+\poin\left(W^{[q]}\right).$$
\end{enumerate}
\end{lemma}

\begin{proof}
For each $d$ we have an exact sequence $0\rightarrow W(d)\rightarrow \oplus O_{C}(d-d_{i})\rightarrow O_{C}(d)$, giving a corresponding exact sequence on global sections. Since\linebreak $H^{0}\!\left(O_{C}(d)\right)$ identifies with $(A/h)_{d}$, the cokernel of the map\linebreak $H^{0}\left(\oplus O_{C}(d\! -\! d_{i})\right)\!\rightarrow H^{0}\left(O_{C}(d)\right)$ identifies with $\left(A/(J,h)\right)_{d}$. It follows that $\dim \left(A/(J,h)\right)_{d}= h^{0}(O_{C}(d))-h^{0}\left(\oplus O_{C}(d-d_{i})\right) + h^{0}(W(d))$. Multiplying by $T^{d}$, summing over $d$, and using (a) and (b) of Example \ref{example1.3}, we get (1). Furthermore, replacing each $g_{i}$ by $g_{i}^{q}$ replaces $J$ by $J^{[q]}$ and $W$ by $W^{[q]}$. So (2) is a consequence of (1).
\end{proof}

\begin{remark}
\label{remark1.6}
Lemma \ref{lemma1.5} allows us to replace the problem of the dependence of $\poin\left(A/\left(J^{[q]},h\right)\right)$ on $q$ by a more geometric question: if $W$ is a vector bundle on $C$, how does $\poin\left(W^{[q]}\right)$ vary with $q$? A generalization of Lemma \ref{lemma1.5} is key to the sheaf-theoretic approach to Hilbert-Kunz theory taken by Brenner and Trivedi.
\end{remark}

For the rest of this section we take $J=(x,y,z)$, $g_{1}=x$, $g_{2}=y$, $g_{3}=z$ so that the $W$ of Lemma \ref{lemma1.5} has rank $2$ and degree $-9$. We'll use sheaf theory on $C$ to give another proof of Pardue's results.

\begin{lemma}
\label{lemma1.7}
$W$ maps onto a rank $1$ degree $-4$ torsion-free sheaf, $M$, whose stalk at the node is the maximal ideal $m$ of the local ring $\newo$.
\end{lemma}

\begin{proof}
$W(1)$ identifies with the kernel of the map $O_{C}\oplus O_{C}\oplus O_{C} \rightarrow O_{C}(1)$ given by $x$, $y$ and $z$. By Lemma 7.1 of \cite{7}, $W(1)$ maps onto a rank $1$ degree $-1$ torsion-free sheaf whose stalk at the node is $m$, and we twist by $O_{C}(-1)$.
\end{proof}

\begin{lemma}
\label{lemma1.8}
Suppose $q=p^{n}$. Let $M$ be the sheaf of Lemma \ref{lemma1.7}. Pull $M$ back by $\Phi^{n} : C\rightarrow C$ and quotient out the maximal torsion subsheaf to get a rank $1$ torsion-free sheaf $M_{n}$. Then $\deg M_{n} = -5q+1$.
\end{lemma}

\begin{proof}
Theorem 2.8 of \cite{7} together with Lemma \ref{lemma1.7} above shows that $\deg M_{n}=\mathrm{constant}\cdot q- \dim\left(\newo/m^{[q]}\right)$. Passing to the completion we find that $\dim\left(\newo/m^{[q]}\right)\linebreak =\dim\left(F[[x,y]]/(xy,x^{q},y^{q})\right)=2q-1$. So $\deg (M_{n})=(\mathrm{constant})\cdot q + 1$. Since $\deg (M) = -4$, the constant is $-5$.
\end{proof}

\begin{lemma}
\label{lemma1.9}
Let $L_{n}$ be the kernel of the obvious map $W^{[q]}\rightarrow M_{n}$. Then:
\begin{enumerate}
\item[(1)] There is an exact sequence $0\rightarrow L_{n}\rightarrow W^{[q]}\rightarrow M_{n}\rightarrow 0$ with $\deg M_{n}=-5q+1$, $\deg L_{n}=-4q-1$.
\item[(2)] Neither $L_{n}$ nor $M_{n}$ is free at the node.
\item[(3)] $\poin\left(W^{[q]}\right)=\poin\left(L_{n}\right)+\poin\left(M_{n}\right)$.
\end{enumerate}
\end{lemma}

\begin{proof}
Since $W^{[q]}$ and $M_{n}$ have degrees $-9q$ and $-5q+1$ we get (1).  If $M_{n}$ is locally free, the exact sequence (1) shows that $L_{n}$ is also. Since we have an exact sequence $0\rightarrow M_{n}^{\supcheck}\rightarrow\left(W^{[q]}\right)^{\supcheck}\rightarrow L_{n}^{\supcheck}\rightarrow 0$ we see conversely that if $L_{n}$ is locally free then so is $M_{n}^{\supcheck\supcheck}=M_{n}$. Suppose now that $L_{n}$ and $M_{n}$ are locally free. Then $q>1$. Let $L_{n}^{\prime}$ and $M_{n}^{\prime}$ be the pull-backs of $L_{n}$ and $M_{n}$ by Frobenius so that we have an exact sequence $0\rightarrow L_{n}^{\prime}\rightarrow W^{[pq]}\rightarrow M_{n}^{\prime}\rightarrow 0$. Then $\deg L_{n+1}-\deg M_{n}^{\prime}= (-4pq-1)-p(-5q+1)=pq-p-1>0$. So the map $L_{n+1}\rightarrow W^{[pq]}/L_{n}^{\prime}=M_{n}^{\prime}$ is the zero-map, and $L_{n+1}\subset L_{n}^{\prime}$. But $\deg L_{n+1}>\deg L_{n}^{\prime}$, and this contradiction establishes (2). Finally, $\deg M_{n}-\deg L_{n}=2-q\le 1$. Combining this with (2) and Lemma \ref{lemma1.4} we get (3).
\end{proof}

\begin{corollary}
\label{corollary1.10}
\begin{eqnarray*}
(1-T)^{-1}\poin\left(W^{[q]}\right) &=& T^{\frac{4q+2}{3}}(1+2T)+T^{\frac{5q+1}{3}}(2+T)\text{ if } q\equiv 1\pod{3}\\
 &=& T^{\frac{4q+1}{3}}(3T)+T^{\frac{5q+2}{3}}(3)\text{ if } q\equiv 2\pod{3}\\
 &=& T^{\frac{4q}{3}}(2T+T^{2})+T^{\frac{5q}{3}}(1+2T)\text{ if } q\equiv 0\pod{3}
\end{eqnarray*}
\end{corollary}

\begin{proof}
Suppose first that $q\equiv 1\pod{3}$. Since $4q+1\equiv 2\pod{3}$, $(1-T)^{-1}\poin (L_{n})\linebreak =T^{\frac{4q+2}{3}}(1+2T)$ by Example \ref{example1.3}~(c). Similarly, since $5q-1\equiv 1\pod{3}$, $(1-T)^{-1}\poin(M_{n}) $ is $T^{\frac{5q+1}{3}}(2+T)$. Now use (3) of Lemma \ref{lemma1.9}. The cases $q\equiv 2\pod{3}$ and $q\equiv 0\pod{3}$ are handled similarly. (When $q\equiv 2\pod{3}$ we use the fact that neither $L_{n}$ nor $M_{n}$ is locally free.)

Now let $e_{n}=\dim \left(A/\left(J^{[q]},h\right)\right)$. Pardue's formula for $e_{n}$ is easily derived from Corollary \ref{corollary1.10}. Let $u_{n}=(1-T)^{-1}\poin \left(W^{[q]}\right)$. By Lemma \ref{lemma1.5}, $(1-T)^{2}\hilb A/\left(J^{[q]},h\right)=\left(1+T+T^{2}\right)\left(1-3T^{q}\right)+u_{n}$. Applying $\left(\frac{\dif}{\dif T}\right)^{2}$, dividing by $2$, and evaluating at $T=1$ we find that $e_{n}=\frac{1}{2}\left(u_{n}^{\prime\prime}(1)-(9q^{2}+9q+4)\right)$. Suppose that $q\equiv 1\pod{3}$. Then Corollary \ref{corollary1.10} shows that $u_{n}^{\prime\prime}(1)=\left(\frac{4q+2}{3}\right)(4q+3)+\left(\frac{5q+1}{3}\right)(5q)=\frac{41}{3}q^{2}+\frac{25}{3}q+2$. When $q\equiv 2\pod{3}$, $u_{n}^{\prime\prime}(1)=\left(\frac{4q+1}{3}\right)(4q+4)+\left(\frac{5q+2}{3}\right)(5q-1)=\frac{41}{3}q^{2}+\frac{25}{3}q+\frac{2}{3}$. And when $q\equiv 0\pod{3}$, $u_{n}^{\prime\prime}(1)=\left(\frac{4q+3}{3}\right)(4q+2)+\left(\frac{5q}{3}\right)(5q+1)=\frac{41}{3}q^{2}+\frac{25}{3}q+2$. So $e_{n}=\frac{7}{3}q^{2}-\frac{1}{3}q-\frac{5}{3}$ if $q\equiv 2\pod{3}$, and is $\frac{7}{3}q^{2}-\frac{1}{3}q-1$ otherwise.
\end{proof}

\section{Elements of $\newz[T,T^{-1}]$ attached to cycles}
\label{section2}

In Corollary \ref{corollary1.10} we calculated all the $(1-T)^{-1}\poin \left(W^{[q]}\right)$ for a certain rank $2$ bundle, $W$. In this section we develop some combinatorial machinery that we'll use later to get similar results for arbitrary $W$.

\begin{definition}
\label{def2.1}
Suppose $r>0$. A cycle (of length $r$) is an ordered $r$-tuple of integers, defined up to cyclic permutation. If $a$ is a cycle, $a(k)$ is the cycle obtained from $a$ by adding $3k$ to each cycle entry.
\end{definition}

\begin{definition}
\label{def2.2}
\begin{itemize}\item[]% to force linebreak
\item[] $\gamma_{1}(a)$ is the number of entries of $a$ that are $\ge 0$.
\item[] $\gamma_{2}(a)=\sum \max (a_{i},0)$, where $a_{i}$ runs over the entries of $a$.
\end{itemize}
\end{definition}

Note that $\gamma_{1}(a)+\gamma_{2}(a)=\sum\max(a_{i}+1,0)$ where $a_{i}$ runs over the entries of $a$. We now compute $(1-T)^{2}\sum\gamma_{2}\left(a(k)\right)T^{k}$. This is evidently a sum of contributions, one for each entry in $a$. An entry of 2 gives a contribution of $(1-T)^{2}(2+5T+8T^{2}+\cdots)=2+T$; similarly an entry of $1$ (resp.\ $0$) gives a contribution of $(1+2T)$ (resp.\ $3T$). If follows easily that an entry of $-n$ gives a contribution of $T^{n+2}\,{3}(2+T)$, $T^{n+1}\,{3}(1+2T)$ or $T^{n}{3}\,(3T)$ according as $n\equiv 1$, $2$ or $0\mod{3}$. We may express this in a slightly different way:

\begin{lemma}
\label{lemma2.3}
Suppose the distinct entries in the cycle $a$ are $-n_{i}$ with $-n_{i}$ appearing $r_{i}$ times in the cycle. Then $P_{2}(a)=(1-T)^{2}\sum\gamma_{2}\left(a(k)\right)T^{k}$ lies in $\newz [T,T^{-1}]$, and is the sum of contributions, one from each $n_{i}$. The contribution from $n_{i}$ is:
\begin{eqnarray*}
\hspace{1.2in}T^{\frac{n_{i}+2}{3}}(2r_{i}+r_{i}T)\qquad&\mathrm{ if }& n_{i}\equiv 1\pod{3}\\
T^{\frac{n_{i}+1}{3}}(r_{i}+2r_{i}T)\qquad&\mathrm{ if }& n_{i}\equiv 2\pod{3}\\
T^{\frac{n_{i}}{3}}(3r_{i}T)\qquad&\mathrm{ if }& n_{i}\equiv 0\pod{3}
\end{eqnarray*}%
\end{lemma}

Observe next that the cycle $a$ gives rise to an integer-valued function of period $r$ on $\newz$, defined up to translation. We say that the cycle is ``aperiodic'' if this function has no period $<r$. For the rest of this section we assume that $r>1$ and that $a$ is aperiodic.

\begin{definition}
\label{def2.4}
A ``bloc'', $b$, of $a$ with entry $N$ consists of consecutive entries of $a$ each of which is $N$, with both the cycle entry preceding the first bloc entry and the cycle entry following the last bloc entry unequal to $N$. The length, $l(b)$, of $b$ is the number of entries in $b$.
\end{definition}

Since $r>1$ and the cycle is aperiodic, there are at least $2$ blocs in $a$. The blocs of $a$ appear in cyclic order and fill out $a$; their lengths sum to $r$.

\begin{definition}
\label{def2.5}
Let $b$ be a bloc with entry $N$.
\begin{enumerate}
\item[(1)] If the blocs just before and just after $b$ have entries $<N$, $b$ is locally maximal and $\varepsilon (b) =1$.
\item[(2)] If the blocs just before and just after $b$ have entries $>N$, $b$ is locally minimal and $\varepsilon (b) =-1$.
\item[(3)] If $b$ is neither locally maximal nor locally minimal, $\varepsilon (b) =0$.
\end{enumerate}
\end{definition}

\begin{remark}
\label{remark2.6}
Between any $2$ locally maximal blocs there is a locally minimal bloc, and between any $2$ locally minimals there is a locally maximal. Since there are at least $2$ blocs, $\sum\varepsilon (b) =0$.
\end{remark}

\begin{definition}
\label{def2.7}
\begin{enumerate}
\item[]%to force linebreak
\item[(1)] A bloc $b$ with entry $N$ is positive if $N\ge 0$.
\item[(2)] Suppose $b$ is positive. $\varepsilon^{*}(b)=\varepsilon(b)$ unless $N=0$ and $b$ is locally maximal. In this case we set $\varepsilon^{*}(b)$ equal to $0$.
\item[(3)] $\gamma_{3}(a)=\sum\varepsilon^{*}(b)$, the sum ranging over the positive blocs of $a$.
\end{enumerate}
\end{definition}

We now compute $(1-T)^{2}\sum\gamma_{3}\left(a(k)\right)T^{k}$. The sum is evidently a sum of contributions, one from each bloc of $a$. Consider first a bloc with entry $2$ or $1$. The contribution of this bloc is $\varepsilon(b)(1-T)^{2}\cdot(1+T+T^{2}+\cdots)=\varepsilon(b)(1-T)$. Next consider a bloc with entry $0$. If the block is locally minimal it gives a contribution of $(-1)(1-T)^{2}(1+T+T^{2}+\cdots)=\varepsilon(b)(1-T)$, while if it is locally maximal, the contribution is $(1)(1-T)^{2}(T+T^{2}+T^{3}+\cdots)=T-T^{2}=\varepsilon(b)\cdot(1-T)-(1-T)^{2}$.

More generally, a locally maximal bloc with entry $-n$, $n\equiv 0\pod{3}$, provides a contribution of $\varepsilon(b)T^{\frac{n}{3}}(1-T)-T^{\frac{n}{3}}(1-T)^{2}$, while in all other cases (i.~e.\ when $n\equiv\pm 1\pod{3}$ or the bloc is not locally maximal) the contribution is $\varepsilon(b)T^{\frac{n+2}{3}}(1-T)$, $\varepsilon(b)T^{\frac{n+1}{3}}(1-T)$, or $\varepsilon(b)T^{\frac{n}{3}}(1-T)$ according as $n\equiv 1$, $2$ or $0 \mod 3$. We'll express this result in a different way.

\begin{definition}
\label{def2.8}
Suppose the distinct entries of $a$ are the integers $-n_{i}$. Then:
\begin{enumerate}
\item[(1)] $s_{i}$ is $\sum\varepsilon(b)$, the sum extending over all the blocs of $a$ with entry $-n_{i}$
\item[(2)] If $n_{i}\equiv 0\pod{3}$, $B_{i}$ is the number of locally maximal blocs with entry $-n_{i}$.
\end{enumerate}
\end{definition}

The discussion preceding the definition shows:

\begin{theorem}
\label{theorem2.9}
$P_{3}(a)=(1-T)^{2}\sum\gamma_{3}\left(a(k)\right)T^{k}$ is a sum of contributions, one from each $n_{i}$. The contribution from $n_{i}$ is:
\begin{eqnarray*}
\hspace{1.5in}T^{\frac{n_{i}+2}{3}}(s_{i}-s_{i}T)\qquad&\mathrm{ if }& n_{i}\equiv 1\pod{3}\\
T^{\frac{n_{i}+1}{3}}(s_{i}-s_{i}T)\qquad&\mathrm{ if }& n_{i}\equiv 2\pod{3}\\
T^{\frac{n_{i}}{3}}\left(s_{i}-s_{i}T-B_{i}(1-T)^{2}\right)\qquad&\mathrm{ if }& n_{i}\equiv 0\pod{3}
\end{eqnarray*}%
\end{theorem}

We next derive an alternative description of $\gamma_{1}(a)+\gamma_{3}(a)$ in terms of ``positive parts of $a$''.

\begin{definition}
\label{def2.10}
A positive part, $p$, of $a$ consists of consecutive entries of $a$ all of which are $\ge 0$; if $a$ has a negative entry we further require that the entry of $a$ preceding the first entry of $p$ and the entry of $a$ following the last entry of $p$ are $<0$. (Note that any positive part of $a$ is a union of consecutive positive blocs.)
\end{definition}

\begin{definition}
\label{def2.11}
\begin{enumerate}
\item[] % to force linebreak
\item[(1)] $\theta(p) = l(p)$ if $p$ consists of a single bloc of zeroes.
\item[(2)] $\theta(p) = l(p)$ if $l(p)=r$.
\item[(3)] In all other cases, $\theta(p)=1+l(p)$.
\end{enumerate}
\end{definition}

\begin{definition}
\label{def2.12}
$\theta(a)=\sum\theta(p)$, the sum extending over the positive parts of $a$.
\end{definition}

\begin{lemma}
\label{lemma2.13}
If $p$ is a positive part of $a$, $\theta(p)=l(p)+\sum\varepsilon^{*}(b)$, the sum extending over the blocs in $p$.
\end{lemma}

\begin{proof}
If $p$ contains a bloc with $\varepsilon^{*}\ne\varepsilon$, then since this bloc is locally maximal with entry $0$ it is the only bloc in $p$ and we use (1) of Definition \ref{def2.11}. So we may assume that $\varepsilon^{*}=\varepsilon$ for each bloc in $p$. If $l(p)=r$, $\sum\varepsilon^{*}(b)=\sum\varepsilon(b)$, which is $0$ by Remark \ref{remark2.6}, and we use (2) of Definition \ref{def2.11}.  Suppose finally that $l(p)<r$. There is at least one bloc in $p$ with $\varepsilon\ne 0$. The first and last blocs appearing in $p$ with $\varepsilon\ne 0$ are evidently locally maximal. The first sentence of Remark \ref{remark2.6} then shows that $\sum\varepsilon(b)$, the sum running over the blocs contained in $p$, is $1$. Definition \ref{def2.11}, (3), now gives the result.
\end{proof}

Summing the result of Lemma \ref{lemma2.13} over the positive parts of $a$ we find:

\begin{corollary}
\label{corollary2.14}
$\theta(a)=\gamma_{1}(a)+\gamma_{3}(a)$.
\end{corollary}

\begin{theorem}
\label{theorem2.15}
Let $\gamma_{4}(a)=\left(\sum\max(a_{i}+1,0)\right)-\theta(a)$ with $\theta(a)$ as in Definition \ref{def2.12}. Let $P_{4}(a)$ be $(1-T)^{2}\cdot\sum\gamma_{4}\left(a(k)\right)T^{k}$. Then $P_{4}(a)$ is a sum of contributions, one from each $n_{i}$, where the $-n_{i}$ are the distinct entries of $a$. In the notation of Lemma \ref{lemma2.3} and Definition \ref{def2.8}, the contribution from $n_{i}$ is:
\begin{eqnarray*}
\hspace{.9in}T^{\frac{n_{i}+2}{3}}((2r_{i}-s_{i})+(r_{i}+s_{i})T)\qquad&\mathrm{ if }& n_{i}\equiv 1\pod{3}\\
T^{\frac{n_{i}+1}{3}}((r_{i}-s_{i})+(2r_{i}+s_{i})T)\qquad&\mathrm{ if }& n_{i}\equiv 2\pod{3}\\
T^{\frac{n_{i}}{3}}\left(-s_{i}+(3r_{i}+s_{i})T+B_{i}(1-T)^{2}\right)\qquad&\mathrm{ if }& n_{i}\equiv 0\pod{3}
\end{eqnarray*}%
\end{theorem}

\begin{proof}
Combining Corollary \ref{corollary2.14} with the sentence following Definition \ref{def2.2} we find that $\gamma_{4}=(\gamma_{1}+\gamma_{2})-(\gamma_{1}+\gamma_{3})=\gamma_{2}-\gamma_{3}$. Applying this to $a(k)$, multiplying by $T^{k}$ and summing over $k$ we find that $P_{4}(a)=P_{2}(a)-P_{3}(a)$. Lemma \ref{lemma2.3} and Theorem \ref{theorem2.9} conclude the proof.
\end{proof}

\section{Results for arbitrary $W$ and $J$}
\label{section3}

A locally free sheaf of rank $>0$ is ``indecomposable'' if it is not a direct sum of two subsheaves of rank $>0$. indecomposable locally free $W$ on the nodal cubic $C$ have been classified --- see Burban \cite{4} and the references given there.  I'll summarize results from the classification.

\begin{enumerate}
\item[(1)] Suppose $r>0$, $a$ is an an aperiodic cycle of length $r$, $m\ge 1$ and $\lambda$ is in $F^{*}$.  One may attach to the triple $a, m, \lambda$ an indecomposable locally free sheaf $W=\newb(a,m,\lambda)$.
\item[(2)] The pull-back of $W$ to $X=\newp^{1}$ is the direct sum of the $\left(O_{X}(a_{i})\right)^{m}$ where the entries of $a$ are the $a_{i}$. In particular, the rank of $W$ is $mr$, and the degree is $m\sum a_{i}$.
\item[(3)] If $W=\newb(a,m,\lambda)$, then $W(k)$ is isomorphic to $\newb(a(k),m,\lambda)$ with $a(k)$ as in Definition \ref{def2.1}.
\item[(4)] When $F$ is algebraically closed (as it is throughout this paper) every indecomposable locally free sheaf on $C$ is isomorphic to some $\newb(a,m,\lambda)$.
\end{enumerate}

In Theorem 2.2 of \cite{6}, Drozd, Greuel and Kashuba give a formula for $h^{0}(W)$ when $W=\newb(a,m,\lambda)$. (As we're dealing with a nodal cubic rather than a cycle of projective lines, we take the $s$ in the statement of that theorem to be $1$.) In particular they show:

\begin{theorem}
\label{theorem3.1}
Suppose $W=\newb(a,m,\lambda)$ with $r>1$. Then in the notation of our section \ref{section2}, $h^{0}(W)=m\cdot\left(\left(\sum\max(a_{i}+1,0)\right)-\theta(a)\right)=m(\gamma_{4}(a))$.
\end{theorem}

\begin{corollary}
\label{corollary3.2}
Situation as in Theorem \ref{theorem3.1}. Then $(1-T)^{-1}\poin(W)=m(1-T)^{2}\sum\gamma_{4}(a(k))T^{k}$.
\end{corollary}

Applying Theorem \ref{theorem2.15} we find:

\begin{theorem}
\label{theorem3.3}
Situation as in Theorem \ref{theorem3.1}.  Suppose the distinct entries in $a$ are $-n_{i}$. Then $(1-T)^{-1}\poin(W)$ is the sum of the following contributions, one from each $n_{i}$:
\begin{eqnarray*}
\hspace{.6in}T^{\frac{n_{i}+2}{3}}\left((2mr_{i}-ms_{i})+(mr_{i}+ms_{i})T\right)\qquad&\mathrm{ if }& n_{i}\equiv 1\pod{3}\\
T^{\frac{n_{i}+1}{3}}\left((mr_{i}-ms_{i})+(2mr_{i}+ms_{i})T\right)\qquad&\mathrm{ if }& n_{i}\equiv 2\pod{3}\\
T^{\frac{n_{i}}{3}}\left(-ms_{i}+(3mr_{i}+ms_{i})T+mB_{i}(1-T)^{2}\right)\qquad&\mathrm{ if }& n_{i}\equiv 0\pod{3}
\end{eqnarray*}%

where $r_{i}$ is the number of times $-n_{i}$ appears in $a$, and $s_{i}$ and $B_{i}$ are obtained from $a$ as in Definition \ref{def2.8}.
\end{theorem}

We now make use of the following key result of Burban \cite{4}: if $W=\newb(a,m,\lambda)$ then $W^{[q]}$ is isomorphic to $\newb(qa,m,\lambda^{q})$ where $qa$ is obtained from $a$ by multiplying each cycle entry, $a_{i}$, by $q$.

\begin{theorem}
\label{theorem3.4}
Let $W$ be a locally free sheaf on $C$. Suppose the pull-back of $W$ to $X=\newp^{1}$ is the direct sum of $\left(O_{X}(-n_{i})\right)^{r_{i}}$ where the $n_{i}$ are distinct and each $r_{i}>0$. Then one can assign to each $n_{i}$ an $s_{i}$ (with $|s_{i}|\le r_{i}$), and to each $n_{i}\equiv 0\pod{3}$ a $B_{i}$, so that the following holds:

For each $q$ (when $p=3$, for each $q>1$), $(1-T)^{-1}\poin\left(W^{[q]}\right)$ is the sum of the following contributions, one for each $n_{i}$:
\begin{eqnarray*}
\hspace{.9in}T^{\frac{qn_{i}+2}{3}}((2r_{i}-s_{i})+(r_{i}+s_{i})T)\qquad&\mathrm{ if }& qn_{i}\equiv 2\pod{3}\\
T^{\frac{qn_{i}+1}{3}}((r_{i}-s_{i})+(2r_{i}+s_{i})T)\qquad&\mathrm{ if }& qn_{i}\equiv 1\pod{3}\\
T^{\frac{qn_{i}}{3}}\left(-s_{i}+(3r_{i}+s_{i})T+B_{i}(1-T)^{2}\right)\qquad&\mathrm{ if }& qn_{i}\equiv 0\pod{3}
\end{eqnarray*}%
\end{theorem}

\begin{proof}
It suffices to prove the result for indecomposable $W$. So we may assume that $W$ is $\newb(a,m,\lambda)$. Suppose first that the length of the cycle $a$ is $>1$. Then $W^{[q]}$ is isomorphic to $\newb\left(qa,m,\lambda^{q}\right)$; furthermore the pull-back of $W^{[q]}$ to $X=\newp^{1}$ is the direct sum of the $\left(O_{X}(-qn_{i}\right)^{mr_{i}}$.

Now replace $W$ by $W^{[q]}$ in Theorem \ref{theorem3.3}. The effect of this is to replace $n_{i}$ by $qn_{i}$ and leave $m$ unchanged. The result we desire would follow if we could show that the $s_{i}$ and $B_{i}$ attached to the cycle $qa$ and its cycle entry $-qn_{i}$ are independent of the choice of $q$ (when $p=3$ we need to show that this independence holds for $q\ge 3$). But as there is an obvious $1$ to $1$ correspondence between the blocs of $a$ and the blocs of $qa$, and this correspondence preserves $\varepsilon$, this is clear.

When the cycle $a$ consists of a single entry, $-n_{1}$, we can make a much simpler argument  In this case $W$ has a filtration with $m$ isomorphic quotients, each a line bundle of degree $-n_{1}$, and it's easy to calculate $(1-T)^{-1}\poin\left(W^{[q]}\right)$. Now $r_{1}=m$, and we find that Theorem \ref{theorem3.4} holds for $W$ with $s_{1}=0$, and when $n_{1}\equiv 0\pod{3}$, $B_{1}=1$ if $\lambda =1$ and $B_{1}=0$ otherwise.
\end{proof}

Suppose now that $W$ is the kernel bundle attached to an ideal $J$ and generators $g_{1},\ldots, g_{s}$ of $J$. Let $d_{i}=\deg g_{i}$, and set $e_{n}=\dim A/\left(J^{[q]},h\right)$ where $q=p^{n}$. Theorem \ref{theorem3.4} attaches to $W$ certain integers $n_{i}$, $r_{i}$, $s_{i}$ and $B_{i}$. We'll use the argument given at the end of section \ref{section1} to express each $e_{n}$ (when $p=3$, each $e_{n}$ with $n>0$) in terms of $n_{i}$, $r_{i}$, $s_{i}$, $B_{i}$ and $\sum d_{i}^{2}$.

\begin{definition}
\label{def3.5}
$\mu=\frac{1}{6}\sum r_{i}n_{i}^{2} - \frac{3}{2}\sum d_{i}^{2}$, $\alpha = \frac{1}{3}\sum s_{i}n_{i}$.
\end{definition}

The general result of Brenner \cite{1} concerning Hilbert-Kunz multiplicities in graded dimension $2$ shows that $e_{n}=\mu q^{2}+O(q)$. We'll show that when $p=3$ (and $n>0$) $e_{n}=\mu q^{2}+\alpha q -R$ for constant $R$. And when $p\ne 3$, $e_{n}=\mu q^{2}+\alpha q -R(q)$ where $R(q)$ only depends on $q\mod 3$.

\begin{theorem}
\label{theorem3.6}
Suppose $p=3$. Let $R=\sum(r_{i}-B_{i})$. Then for $n>0$, $e_{n}=\mu q^{2}+\alpha q-R$.
\end{theorem}

\begin{proof}
Let $u_{n}=(1-T)^{-1}\poin\left(W^{[q]}\right)$ and $v_{n}=\left(1\! +T\! +T^{2}\right)\cdot\left(-1\! +\! \sum T^{d_{i}q}\right)$. As we saw in section \ref{section1}, $2e_{n}=u_{n}^{\prime\prime}(1)-v_{n}^{\prime\prime}(1)$; see Lemma \ref{lemma1.5} and the proof of Corollary \ref{corollary1.10}. Now $v_{n}^{\prime\prime}(1)=-2 + $ a sum of terms $(d_{i}q)(d_{i}q-1)+(d_{i}q+1)(d_{i}q)+(d_{i}q+2)(d_{i}q+1)$. Expanding we find that $v_{n}^{\prime\prime}(1)=\left(3\sum d_{i}^{2}\right)q^{2}+\left(3\sum d_{i}\right)q+2s-2$, where $s$ is the number of $q_{i}$. Since $W$ has degree $-\sum r_{i}n_{i}$ and rank $\sum r_{i}$ we find:

\begin{equation}\tag{*}
\label{star}
v_{n}^{\prime\prime}(1)=\left(3\sum d_{i}^{2}\right)q^{2}+\left(\sum r_{i}n_{i}\right)q+2\sum r_{i}
\end{equation}

Now as $p=3$ and $q>1$, each $qn_{i}\equiv 0 \pod{3}$. Theorem \ref{theorem3.4} then shows that $u_{n}$ is a sum of terms $T^{\frac{qn_{i}}{3}}\left(-s_{i}+\left(3r_{i}+s_{i}\right)T+B_{i}(1-T)^{2}\right)$. So $u_{n}^{\prime\prime}(1)$ is a sum of terms $\frac{qn_{i}}{3}\cdot\frac{qn_{i}-3}{3}\cdot(-s_{i})+\frac{qn_{i}+3}{3}\cdot\frac{qn_{i}}{3}\cdot\left(3r_{i}+s_{i}\right)+2B_{i}$. This term simplifies to $\frac{qn_{i}}{3}\left(qr_{i}n_{i}+3r_{i}+2s_{i}\right)+2B_{i}$, and so:

\begin{equation*}\tag{**}
\label{starstar}
u_{n}^{\prime\prime}(1)=\left(\frac{1}{3}\sum r_{i}n_{i}^{2}\right)q^{2}+\left(\sum r_{i}n_{i}\right)q+\left(\frac{2}{3}\sum s_{i}n_{i}\right)q+2\sum B_{i}.
\end{equation*}

Combining (\ref{star}) and (\ref{starstar}) we find that $2e_{n}=u_{n}^{\prime\prime}(1)-v_{n}^{\prime\prime}(1)=2\mu q^{2}+2\alpha q +2\sum\left(B_{i}-r_{i}\right)$, giving the theorem.
\end{proof}

\begin{theorem}
\label{theorem3.7}
Suppose $p\ne 3$. Set $$\textstyle R(q)=\sum_{qn_{i}\equiv 1\pod{3}}\left(\frac{2r_{i}-2s_{i}}{3}\right)+\sum_{qn_{i}\equiv 2\pod{3}}\left(\frac{2r_{i}-s_{i}}{3}\right)+\sum_{qn_{i}\equiv 0\pod{3}}\left(r_{i}-B_{i}\right).$$ Note that $R(q)$ only depends on $q\mod 3$. Then $e_{n}=\mu q^{2}+\alpha q-R(q)$.
\end{theorem}

\begin{proof}
We argue as in the proof of Theorem \ref{theorem3.6}. (\ref{star}) remains valid, but now $u_{n}^{\prime\prime}(1)$ is a more complicated sum of terms. When $qn_{i}\equiv 0\pod{3}$, the term once again is $\frac{qn_{i}}{3}\left(qr_{i}n_{i}+3r_{i}+2s_{i}\right)+2B_{i}$. But when $qn_{i}\equiv 1\pod{3}$ this term is replaced by $\frac{qn_{i}+2}{3}\left(qr_{i}n_{i}+r_{i}+2s_{i}\right)$; that is to say by $\frac{qn_{i}}{3}\left(qr_{i}n_{i}+3r_{i}+2s_{i}\right)+\frac{2r_{i}+4s_{i}}{3}$. And when $qn_{i}\equiv 2\pod{3}$, it is replaced by $\frac{qn_{i}+1}{3}\left(qr_{i}n_{i}+2r_{i}+2s_{i}\right)$; that is to say by $\frac{qn_{i}}{3}\left(qr_{i}n_{i}+3r_{i}+2s_{i}\right)+\frac{2r_{i}+2s_{i}}{3}$.  So:
\begin{equation*}
u_{n}^{\prime\prime}(1)=\left(\frac{1}{3}\sum r_{i}n_{i}^{2}\right)q^{2}+\left(\sum r_{i}n_{i}\right)q +\hspace{-1em}\sum_{qn_{i}\equiv 1\pod{3}}\hspace{-1em}\frac{2r_{i}+4s_{i}}{3}+\hspace{-1em}\sum_{qn_{i}\equiv 2\pod{3}}\hspace{-1em}\frac{2r_{i}+2s_{i}}{3}+2\sum B_{i}.
\end{equation*}

Combining the above result with (\ref{star}) we find that $2e_{n}=u_{n}^{\prime\prime}(1)-v_{n}^{\prime\prime}(1)=2\mu q^{2}+2\alpha q +\sum_{qn_{i}\equiv 1\pod{3}}\frac{4s_{i}-4r_{i}}{3}+\sum_{qn_{i}\equiv 2\pod{3}}\frac{2s_{i}-4r_{i}}{3}+2\sum_{qn_{i}\equiv 0\pod{3}} \left(B_{i}-r_{i}\right)=2\mu q^{2}+2\alpha q-2R(q)$.
\end{proof}

Theorems \ref{theorem3.6} and \ref{theorem3.7} differ from similar results in \cite{2} and \cite{7} in that they allow practical calculation of all the $e_{n}$ (The eventually periodic terms that occur in the results of \cite{2} and \cite{7} arise from dynamical systems acting on the rational points of certain moduli spaces --- in practice they cannot be calculated.) The following examples show how easy it is to apply Theorems \ref{theorem3.6} and \ref{theorem3.7}.

\begin{example}
\label{example3.8}
Suppose $p=2$ and $h=x^{3}+y^{3}+xyz$. Let $J$ be generated by $g_{1},\ldots, g_{8}$ where the $g_{i}$ are $x^{3}$, $y^{3}$, $z^{3}$, $x^{2}y$, $x^{2}z$, $xz^{2}$, $y^{2}z$ and $yz^{2}$. If $W$ is the kernel bundle arising from these $g_{i}$, then $(1-T)^{-1}\poin\left(W^{[8]}\right)=(1-T)^{-1}\poin\left(A/\left(J^{[8]},h\right)\right)-\left(1-T^{3}\right)\left(1-8T^{24}\right)$. This is calculated immediately using Macaulay 2 which shows: 
\begin{equation*}
(1-T)^{-1}\poin\left(W^{[8]}\right)=3T^{27}+12T^{28}+6T^{30}=T^{27}(3+12T)+T^{30}(6+0T).
\end{equation*}

We'll use this information to determine all the $e_{n}$.

\begin{enumerate}
\item[(a)]$n_{1}=\lfloor\frac{3\cdot 27}{8}\rfloor = 10$\qquad$n_{2}=\lfloor\frac{3\cdot 30}{8}\rfloor = 11$
\item[(b)] Since $8n_{1}\equiv 2\pod{3}$, $r_{1}-s_{1}=3$ and $2r_{1}+s_{1}=12$. It follows that $r_{1}=5$, $s_{1}=2$. Similarly, since $8n_{2}\equiv 1\pod{3}$,  $2r_{2}-s_{2}=6$ and $r_{2}+s_{2}=0$. So $r_{2}=2$, $s_{2}=-2$.
\item[(c)]$\mu = \frac{1}{6}(5\cdot 100+2\cdot 121)-\frac{3}{2}\left(\sum_{1}^{8}9\right)=\frac{47}{3}$\\
$\alpha=\frac{1}{3}(2\cdot 10 - 2\cdot 11)=-\frac{2}{3}$
\item[(d)]Since $n_{1}\equiv 1\pod{3}$ and $n_{2}\equiv 2\pod{3}$,\\ 
$R(1) = \frac{2r_{1}-2s_{1}}{3}+\frac{2r_{2}-s_{2}}{3} = \frac{6}{3}+\frac{6}{3} = 4$\\
$R(2) = \frac{2r_{1}-s_{1}}{3}+\frac{2r_{2}-2s_{2}}{3} = \frac{8}{3}+\frac{8}{3} = \frac{16}{3}$
\end{enumerate}

Theorem \ref{theorem3.7} now tells us that $e_{n}=\frac{47}{3}q^{2}-\frac{2}{3}q-4$ for even $n$ and $\frac{47}{3}q^{2}-\frac{2}{3}q-\frac{16}{3}$ for odd $n$. 
\end{example}

\begin{example}
\label{example3.9}
Take the $g_{i}$ and $h$ as in the above example but with $p=3$. Now Macaulay 2 gives:
\begin{eqnarray*}
(1-T)^{-1}\poin\left(W^{[9]}\right)&=&13T^{31}+2T^{32}+2T^{33}+4T^{34}\\
&=& T^{30}(0+13T+2T^{2})+T^{33}(2+4T+0T^{2})
\end{eqnarray*}

It follows that $n_{1}=\frac{30\cdot 3}{9}=10$, and we find that $r_{1}=5$, $s_{1}=2$, $B_{1}=2$. Similarly, $n_{2}=\frac{33\cdot 3}{9}=11$, and $r_{2}=2$, $s_{2}=-2$, $B_{2}=0$. The $\mu$ and $\alpha$ are once again $\frac{47}{3}$ and $-\frac{2}{3}$, but now $R=(5-2)+(2-0)=5$. We conclude from Theorem \ref{theorem3.6} that $e_{n}=\frac{47}{3}q^{2}-\frac{2}{3}q-5$ for $n>0$.
\end{example}

%%%%%%%%
%%%%%%%%%%%%%%%%

\label{}

% \appendix

% \section{}
% \label{}

\end{document}